\documentclass{amsart}
\usepackage[utf8]{inputenc}

\usepackage{amssymb,tikz}
\usepackage{latexsym}
\usepackage{amsmath,mathdots}
\usepackage{euscript}
\usepackage{graphics}
\usepackage{todonotes}
\usepackage{pgf,tikz,pgfplots}
\pgfplotsset{compat=1.14}
\usepackage{mathrsfs}
\usetikzlibrary{arrows}

\newtheorem*{definition}{Definition}
\newtheorem*{conjecture}{Conjecture}

\newtheorem{theorem}{Theorem}
\newtheorem{proposition}{Proposition}

\begin{document}

\date{\today}
\title{Polynomial atlases on manifolds}

\author[N.~Juricic]{Nicholas~Juricic}
\address{
NJ,
Department of Mathematics\\
University of Arizona\\
617 N Santa Rita Avenue, P.O. Box 210089\\
Tucson, AZ 85721-1009, USA}
\email{juricic@math.arizona.edu}

%%%%%%%%%%%%%%%%%%%%%%%%%%%%%%%%%%%%%%%%%%%%%%%%%%%%%%%%%%%%%%%%%%%%%%%%%%%%%%%%%%

\maketitle

\begin{abstract}
    We consider manifolds whose transition maps are restrictions of polynomial mappings $\mathbb{R}^n\to\mathbb{R}^n$, and use them to give an equivalent statement of the Jacobian conjecture over the real field.
\end{abstract}

\section{Introduction} \label{sec:intro}

Let $n$ be a positive integer. If $F: \mathbb{R}^n \to \mathbb{R}^n$ is a polynomial\footnote{A polynomial function $\mathbb{R}^n \to \mathbb{R}^n$ has polynomial component functions $\mathbb{R}^n \to \mathbb{R}$.} automorphism, then the Jacobian determinant of $F$ is identically equal to some real number $c \neq 0$, and we write $|J_F| \equiv c$. The converse of this statement is precisely the Jacobian conjecture over $\mathbb{R}$.
 
\begin{conjecture}
If $F : \mathbb{R}^n \to \mathbb{R}^n$ is a polynomial satisfying $|J_F| \equiv c$ for some $c \neq 0$, then $F$ is invertible and $F^{-1}$ is a polynomial.
\end{conjecture}

Despite substantial efforts, the Jacobian conjecture remains an open problem in algebraic geometry, and has diverse connections to different areas of mathematics. For example, the conjecture can be viewed in the context of group theory.

\begin{proposition} \label{prop:group}
Let $X_n$ denote the set of all polynomials $F: \mathbb{R}^n \to \mathbb{R}^n$ such that $|J_F| \equiv c$ for some $c \neq 0$. The set $X_n$ forms a group under function composition if and only if the Jacobian conjecture is true over $\mathbb{R}$. 
\end{proposition}

The purpose of this paper is to elaborate such a connection with the theory of manifolds, as follows: consider $\mathbb{R}^n$ as a topological manifold, and define a collection $\mathcal{A}_n$ of pairs of the form $(\mathbb{R}^n, F)$, where $F : \mathbb{R}^n \to \mathbb{R}^n$ is a polynomial satisfying $|J_F| \equiv c$ for some $c \neq 0$. Our main result is the following equivalence.

\begin{theorem} \label{thm:manifold}
The pair $(\mathbb{R}^n, \mathcal{A}_n)$ determines a polynomial manifold if and only if the Jacobian conjecture is true over $\mathbb{R}$.
\end{theorem}

In what follows, we recall precise definitions of topological and differentiable manifolds, and use them to motivate a definition of polynomial manifolds. We then provide a proof of Proposition~\ref{prop:group} and use it to prove Theorem~\ref{thm:manifold}.

\section{Definitions and results}

A topological $n$-manifold is a topological space that is both Hausdorff and second countable, and in which each point has a neighborhood homeomorphic to $\mathbb{R}^n$. We define a chart on a topological $n$-manifold $M$ as a pair $(U, \varphi)$, where $U$ is an open subset of $M$ and $\varphi : U \to \mathbb{R}^n$ is a homeomorphism. An atlas on $M$ is a collection of charts $\mathcal{A} = \{(U_\alpha, \varphi_\alpha)\}$ on $M$ such that the collection $\{U_\alpha\}$ covers $M$. 

An atlas $\mathcal{A}$ is called differentiable whenever the following condition is satisfied: for any pair of charts $(U,\varphi)$ and $(V, \psi)$ in $\mathcal{A}$, either $U \cap V$ is empty, or the maps
\begin{equation} \label{eqn:transition}
    \psi \circ \varphi^{-1} : \varphi(U \cap V) \to \psi(U \cap V)\quad\text{and}\quad \varphi \circ \psi^{-1} : \psi(U \cap V) \to \varphi(U \cap V)
\end{equation}
are differentiable as mappings between open subsets of $\mathbb{R}^n$. Two differentiable atlases $\mathcal{A}$ and $\mathcal{A}'$ on $M$ are said to be compatible whenever each $(U, \varphi) \in \mathcal{A}$ and $(V, \psi) \in \mathcal{A}'$ satisfy the condition (\ref{eqn:transition}). By Zorn's lemma, any differentiable atlas $\mathcal{A}$ is contained within a unique compatible differentiable atlas $\overline{\mathcal{A}}$ such that $\overline{\mathcal{A}}$ is maximal among all differentiable atlases compatible with $\mathcal{A}$, in the sense that if $\mathcal{A}$ and $\mathcal{A}'$ are compatible, then $\mathcal{A}' \subseteq \overline{\mathcal{A}}$.  A differentiable manifold is a pair $(M, \overline{\mathcal{A}})$, where $M$ is a topological manifold and $\overline{\mathcal{A}}$ is a maximal differentiable atlas on $M$. Therefore, given a differentiable atlas $\mathcal{A}$ on $M$, we may speak of the differentiable manifold determined by the pair $(M,\mathcal{A})$.

The notion of a $C^k$-manifold coincides with that of a differentiable manifold, except that the maps in (\ref{eqn:transition}) are required to be $k$-times differentiable. In the same way, the definition of a smooth manifold requires these maps to be smooth, a real analytic manifold requires them to be real analytic, and a polynomial manifold requires them to be polynomial. Further elaboration on terminology surrounding manifolds can be found in the many textbooks on manifold theory, such as \cite{Lee}.

\begin{definition}[Polynomial atlas]
A polynomial atlas $\mathcal{A}$ on a topological $n$-manifold $M$ is an atlas on $M$ such that the following condition is satisfied: for any charts $(U, \varphi)$ and $(V, \psi)$ in $\mathcal{A}$, either $U \cap V$ is empty, or 
\begin{equation} \label{eqn:polynomialcompatibility}
    \psi \circ \varphi^{-1} : \varphi(U \cap V) \to \psi(U \cap V)\quad\text{and}\quad \varphi \circ \psi^{-1} : \psi(U \cap V) \to \varphi(U \cap V)
\end{equation}
are restrictions of polynomials; that is, there exist polynomials $F, G : \mathbb{R}^n \to \mathbb{R}^n$ such that $\psi \circ \varphi^{-1} = F\rvert_{\varphi(U\cap V)}$ and $\varphi \circ \psi^{-1} = G\rvert_{\psi(U \cap V)}.$
\end{definition}

By analogy with differentiable manifolds, if $(M, \mathcal{A})$ is a topological $n$-manifold, and $\mathcal{A}'$ is a polynomial atlas on $M$, we say $\mathcal{A}$ and $\mathcal{A}'$ are compatible provided condition (\ref{eqn:polynomialcompatibility}) is satisfied for any charts $(U,\varphi) \in \mathcal{A}$ and $(V,\psi) \in \mathcal{A}'$. Again, any polynomial atlas on $M$ is contained in a maximal compatible polynomial atlas $\overline{\mathcal{A}}$. 

\begin{definition}[Polynomial manifold]
A polynomial $n$-manifold is a pair $(M, \overline{\mathcal{A}})$, where $M$ is a topological $n$-manifold, and $\overline{\mathcal{A}}$ is a maximal polynomial atlas on $M$. 
\end{definition}

Polynomial manifolds have been studied in \cite{Bieber} and \cite{Sabitova}, motivated in connection to the group of polynomial automorphisms of $\mathbb{R}^n$.  A \textit{polynomial automorphism} of $\mathbb{R}^n$ is an invertible polynomial $F: \mathbb{R}^n \to \mathbb{R}^n$ such that $F^{-1}$ is again a polynomial. While this group's elements have not (yet) been determined, the Jacobian conjecture over $\mathbb{R}$ asserts they are precisely the polynomials $F: \mathbb{R}^n \to \mathbb{R}^n$ such that $|J_F| \equiv c$ for some $c \neq 0$. Let us now give a proof of Proposition~\ref{prop:group}.

\begin{proof}[Proof of Proposition~\ref{prop:group}]
Consider the set $X_n$ of polynomials $F : \mathbb{R}^n \to \mathbb{R}^n$ satisfying $|J_F| \equiv c$ for some $c \neq 0$. The multivariate chain rule and $\det(AB) = \det(A)\det(B)$ imply $X_n$ is closed under composition. Function composition is associative. The identity map $I_n : \mathbb{R}^n \to \mathbb{R}^n$ is a polynomial and $|J_{I_n}| \equiv 1$, hence $I_n \in X_n$, so $X_n$ has an identity.  It follows that $X_n$ forms a group under composition if and only if $F \in X_n$ implies the existence of an inverse for $F$ in $X_n$, i.e. $F^{-1}$ exists and is a polynomial. Note: the inverse function theorem implies that if $|J_F| \equiv c \neq 0$, then $|J_{F^{-1}}| \equiv 1/c$, provided $F^{-1}$ exists.

$(\implies)$ If $X_n$ forms a group under function composition, then any polynomial $F: \mathbb{R}^n \to \mathbb{R}^n$ with constant nonzero Jacobian determinant is invertible and has a polynomial inverse, which is the statement of the Jacobian conjecture over $\mathbb{R}$. 

$(\impliedby)$ Conversely, if the Jacobian conjecture is true over $\mathbb{R}$, then $F \in X_n$ implies $F^{-1}$ exists and is a polynomial with constant Jacobian determinant, so $F^{-1} \in X_n$. 
\end{proof}

It remains to prove Theorem~\ref{thm:manifold}. Let $X_n$ be defined as in Proposition~\ref{prop:group}, and recall $\mathcal{A}_n$ is the collection of pairs $(\mathbb{R}^n, F)$, where $F: \mathbb{R}^n \to \mathbb{R}^n$ is a polynomial satisfying $|J_F| \equiv c$ for some $c \neq 0$.

\begin{proof}[Proof of Theorem~\ref{thm:manifold}]
$(\impliedby)$ Suppose the Jacobian conjecture is true over $\mathbb{R}$. We argue that $\mathcal{A}_n$ is a polynomial atlas on $\mathbb{R}^n$, so is contained in a maximal polynomial atlas, and determines a polynomial manifold structure on $\mathbb{R}^n$. Let $(\mathbb{R}^n, F)$ and $(\mathbb{R}^n, G)$ be any pairs in $\mathcal{A}_n$. By the Jacobian conjecture over $\mathbb{R}$, it follows that $F^{-1}$ and $G^{-1}$ exist and are polynomials. All polynomials are continuous, so $(\mathbb{R}^n, F)$ and $(\mathbb{R}^n, G)$ are actually \textit{charts} on $\mathbb{R}^n$, and $\mathcal{A}_n$ is an atlas on $\mathbb{R}^n$. Furthermore, $$F \circ G^{-1} : \mathbb{R}^n \to \mathbb{R}^n\quad\text{and}\quad G \circ F^{-1} : \mathbb{R}^n \to \mathbb{R}^n$$ are again polynomials, implying that $\mathcal{A}_n$ is a polynomial atlas on $\mathbb{R}^n$. 

$(\implies)$ Suppose the pair $(\mathbb{R}^n, \mathcal{A}_n)$ determines a polynomial manifold, which is to say $\mathcal{A}_n$ is a polynomial atlas on $\mathbb{R}^n$. This means any pairs $(\mathbb{R}^n, F)$ and $(\mathbb{R}^n, G)$ in $\mathcal{A}_n$ are charts, so $F$ and $G$ are in particular homeomorphisms. It therefore makes sense to consider the set $X_n$ as a subset of the group $\mathcal{H}$ of all homeomorphisms $\mathbb{R}^n \to \mathbb{R}^n$. We claim that $X_n$ is a subgroup of $\mathcal{H}$. To see this, first note that $I_n \in X_n$, so $X_n$ is nonempty; and, if $F \in X_n$ and $G \in X_n$, then $F \circ G^{-1} : \mathbb{R}^n \to \mathbb{R}^n$ also lies in $X_n$, by our assumption that $\mathcal{A}_n$ is a polynomial atlas on $\mathbb{R}^n$. We conclude that $X_n$ forms a group under composition, so Proposition~\ref{prop:group} implies the Jacobian conjecture over $\mathbb{R}$. 
\end{proof}

\subsection*{Acknowledgements}
Gratitude to Tyler Kline, Alexis Vizzerra, Tristan Phillips, Joseph Ruiz, Anton Izosimov, Akshita Sharma, Sam Nasreldine, and Doug Haessig for their interest, support, and feedback.
\nocite{Nash}

\nocite{Smith}
\bibliography{bib.bib}
\bibliographystyle{abbrv}

\end{document}